\documentclass[10pt]{article}
\font\smallit=cmti10
\font\smalltt=cmtt10

\usepackage{amssymb,latexsym,amsmath,epsfig,amsthm,graphicx} 

\makeatletter

\renewcommand\section{\@startsection {section}{1}{\z@}
{-30pt \@plus -1ex \@minus -.2ex}
{2.3ex \@plus.2ex}
{\normalfont\normalsize\bfseries\boldmath}}

\renewcommand\subsection{\@startsection{subsection}{2}{\z@}
{-3.25ex\@plus -1ex \@minus -.2ex}
{1.5ex \@plus .2ex}
{\normalfont\normalsize\bfseries\boldmath}}

\renewcommand{\@seccntformat}[1]{\csname the#1\endcsname. }

\makeatother

\newtheorem{lemma}{Lemma}
\newtheorem{proposition}{Proposition}

\theoremstyle{definition}

\begin{document}

\begin{center}
\uppercase{\bf \boldmath A solution to the Straus Erd\H{o}s Conjecture}
\vskip 20pt
{\bf Kyle Bradford}\\
{\smallit Unaffiliated}\\
{\tt kyle.bradford@gmail.com}\\
\end{center}
\vskip 20pt
\centerline{\smallit Received: , Revised: , Accepted: , Published: } 
\vskip 30pt

\centerline{\bf Abstract}
\noindent
This paper outlines a solution to the Straus Erd\H{o}s Conjecture.   Namely for each prime $p$  there exists positive integers $x \leq y \leq z$  so that $$ \frac{4}{p} = \frac{1}{x} + \frac{1}{y} + \frac{1}{z}. $$

\pagestyle{myheadings}
\markright{\smalltt INTEGERS: 24 (2024)\hfill}
\thispagestyle{empty}
\baselineskip=12.875pt
\vskip 30pt

\section{Part One}

For years, I have been waiting for this moment when I solved the problem.  May it bless us all with its simplicity and ease.The Straus-Erd\H{o}s conjecture has been studied and this paper gives an elementary proof.  We are seeking solutions to the diophantine equation mentioned in the abstract.

It should be clear that only one solution exists when the prime $p=2$.  This solution is $$ \frac{4}{2} = \frac{1}{1} + \frac{1}{2} + \frac{1}{2}.$$

It should also be clear that a variety of solutions exist when the prime $p \equiv 3 \bmod 4$.  This solution is derived from the greedy algorithm where you can make two unit fractions: $$ \frac{4}{p} = \frac{4}{p+1} + \frac{4}{p(p+1)}.$$

It has been shown that $p \nmid x$, $p \mid z$  and $p$  sometimes divides $y$.  The common nomenclature is to call solutions with $p \nmid y$  Type I solutions and solutions with $p \mid y$  Type II solutions.  The following propositions and lemmata will elucidate the method of proof.  We start with Type I solutions.

\begin{proposition}
Given a prime $p \equiv 1 \bmod 4$, if a Type I solutions exists, then there exists a nonnegative integer $k$  so that $$  z= \frac{(4k+3)p^{2}+p}{4}.$$
\end{proposition}

\begin{proof}
Let $p$  be a prime so that $p \equiv 1 \bmod 4$.  Suppose that a Type 1 solution exists with $x \leq y \leq z$.

We have from a previous paper that if a solution exists, then $$z =  \frac{xyp}{\gcd(xy,x+y)}$$

\noindent which would imply that $4xy - (x+y)p = \gcd(xy,x+y)$.  Notice then that we can write $$xyp = \frac{(x+y)p^{2} + \gcd(xy, x+y)p}{4}.$$

This would imply that $$ z=\frac{\left( (x+y) \slash \gcd(xy,x+y) \right)p^{2}+p}{4}. $$

We know that $z$  is an integer and $p \equiv 1 \bmod 4$, so $(x+y) \slash \gcd(xy,x+y) \equiv 3 \bmod 4$.    This implies that there exists a nonnegative $k$  so that $(x+y) \slash \gcd(xy,x+y) = 4k+3$.  This shows that there exists a nonnegative integer $k$  so that $$  z= \frac{(4k+3)p^{2}+p}{4}.$$
\end{proof}

Given this prime $p$  and value of $k$,  unique solutions are determined by the different ways that we can express the following fraction as the sum of two positive unit fractions:  $$ \frac{4k+3}{\left( ((4k+3)p+1) \slash 4 \right)}. $$

Instead of supposing that a solution exists, we will use this form to outline for which primes a solution exists. 

\begin{lemma}
Let $k \geq 0$, $1 \leq \ell \leq 2(4k+3)$  so that $\gcd(\ell, 4k+3) = 1$  and consider primes of the form $p \equiv n \bmod (16 \cdot \ell \cdot (4k+3) - 4 \cdot \ell^{2}) \slash (\gcd(\ell,4))^{2}$,  where $n$  is a positive integer so that $(4k+3)n \equiv -1 \bmod (16 \cdot \ell \cdot (4k+3) - 4 \cdot \ell^{2}) \slash (\gcd(\ell,4))^{2}$, then there are solutions $$ \frac{4}{p} = \frac{4(4k+3)-\ell}{(4k+3)p+1} + \frac{\ell}{(4k+3)p+1}  + \frac{4}{p((4k+3)p+1)}  .$$
\end{lemma}

It should be clear that for primes of these forms, these three fractions will be unit fractions.  In fact, they are designed to be unit fractions in a optimal way.  The limitation on $\ell$  guarantees that $x \leq y \leq z$  as read from left to right.  We next move to Type II solutions for sake of symmetry.

\begin{proposition}
Given a prime $p$, if a Type II solutions exists, then there exists a positive integer $k$  so that $$  x= \frac{p+(4k+3)}{4}.$$
\end{proposition}

The proof of this is trivial given my previous work because we showed that integer $x \geq \left\lceil \frac{p}{4} \right\rceil$. This time, given this prime $p$  and value of $k$,  unique solutions are determined by the different ways that we can express the following fraction as the sum of two positive unit fractions:  $$ \frac{4k+3}{p\left( (p+(4k+3)) \slash 4 \right)}. $$

We will use this form to outline for which primes a solution exists. 

\begin{lemma}
Let $k \geq 0$, $1 \leq \ell \leq 2(4k+3)$   so that $\gcd(\ell, 4k+3) = 1$  and consider primes of the form $p \equiv  -(4k+3) \bmod (16 \cdot \ell \cdot (4k+3) - 4 \cdot \ell^{2}) \slash (\gcd(\ell,4))^{2}$, then there are solutions $$ \frac{4}{p} = \frac{4}{p+(4k+3)}  + \frac{4(4k+3) - \ell}{p(p+(4k+3))} + \frac{\ell}{p(p+(4k+3))}.$$
\end{lemma}

Again it should be clear that these are unit fractions.  Combining these two lemmata will help us derive a covering system.  Remember we are only considering primes $p \equiv 1 \bmod 4$.  For each $k \geq 0$  and $1 \leq \ell \leq 2(4k+3)$  such that $\gcd(\ell, 4k+3)=1$, we will find negative $4k+3$  and the negative inverse element of $4k+3$ in the modular group $\mathbb{Z} \slash ((16 \cdot \ell \cdot (4k+3) - 4 \cdot \ell^{2}) \slash (\gcd(\ell,4))^{2}) \mathbb{Z}$ and this will create our covering system.  To clarify this idea, I will consider solutions for all such $\ell$ when $k=0$.

For primes $p \equiv 29 \bmod 44$  we have solutions of the form $$ \frac{4}{p} = \frac{11}{3p+1} + \frac{1}{3p+1} + \frac{4}{p(3p+1)},$$

\noindent for primes $p \equiv 41 \bmod 44$  we have solutions of the form $$ \frac{4}{p} = \frac{4}{p+3} + \frac{11}{p(p+3)} + \frac{1}{p(p+3)},$$

\noindent for primes $p \equiv 13 \bmod 20$  we have solutions of the form $$ \frac{4}{p} = \frac{10}{3p+1} + \frac{2}{3p+1} + \frac{4}{p(3p+1)},$$

\noindent for primes $p \equiv 17 \bmod 20$  we have solutions of the form $$ \frac{4}{p} = \frac{4}{p+3} + \frac{10}{p(p+3)} + \frac{2}{p(p+3)},$$

\noindent for primes $p \equiv 5 \bmod 8$  we have solutions of the form $$ \frac{4}{p} = \frac{8}{3p+1} + \frac{4}{3p+1} + \frac{4}{p(3p+1)},$$

\noindent for primes $p \equiv 5 \bmod 8$  we have solutions of the form $$ \frac{4}{p} = \frac{4}{p+3} + \frac{8}{p(p+3)} + \frac{4}{p(p+3)},$$

\noindent for primes $p \equiv 93 \bmod 140$  we have solutions of the form $$ \frac{4}{p} = \frac{7}{3p+1} + \frac{5}{3p+1} + \frac{4}{p(3p+1)},$$

\noindent for primes $p \equiv 137 \bmod 140$  we have solutions of the form $$ \frac{4}{p} = \frac{4}{p+3} + \frac{7}{p(p+3)} + \frac{5}{p(p+3)}.$$

The last thing that we must show is that this is a covering system.  From this I notice that $5, 13, 17, 29$  are the first four primes that are $1$ modulo $4$.

\bibliographystyle{amsplain}

\end{document}